\documentclass[12 pt, reqno]{amsart}
\usepackage{amsmath}
\usepackage{enumerate}
\usepackage{amsfonts}
\usepackage{amssymb}
\usepackage{amsthm}
\usepackage{fullpage,color}
\usepackage{todonotes}
\usepackage{mathptmx}
\usepackage{hyperref}
\usepackage{cleveref}
\usepackage{comment}

\newtheorem{theorem}{Theorem}
\newtheorem{prop}[theorem]{Proposition}
\newtheorem{lem}[theorem]{Lemma}
\newtheorem{lemma}[theorem]{Lemma}

\theoremstyle{definition}
\newtheorem{definition}[theorem]{Definition}

\newtheorem{question}{Question}

\newtheorem{ex}[theorem]{Example}
\newtheorem{example}[theorem]{Example}

\newtheorem{remark}[theorem]{Remark}

\DeclareMathOperator{\Spec}{Spec}
\DeclareMathOperator{\C}{\text{\u C}}
\DeclareMathOperator{\height}{ht}

\DeclareMathOperator{\grade}{grade}
\DeclareMathOperator{\depth}{depth}
\DeclareMathOperator{\pgrade}{p-grade}
\DeclareMathOperator{\Min}{Min}

\DeclareMathOperator{\Ass}{Ass}
\DeclareMathOperator{\ara}{ara}

\newcommand{\m}{\mathfrak{m}}
\newcommand{\n}{\mathfrak{n}}
\newcommand{\p}{\mathfrak{p}}
\newcommand{\K}{{K}}

\author{Youngsu Kim}
\address{Department of Mathematics, University of Arkansas, Fayetteville,
Arkansas 72701, U.S.A.}
\email{yk009@uark.edu}

\author{Andrew Walker}
\address{Department of Mathematics, College of Charleston, Charleston, South Carolina 29424, U.S.A.}
\email{walkeraj@cofc.edu}

\title{A note on Non-Noetherian Cohen-Macaulay rings}
\date{\today}


\begin{document}
\maketitle
\section{Introduction}

In this note, we study the Cohen-Macaulayness of non-Noetherian rings. 
We show that Hochster's celebrated theorem that a finitely generated normal semigroup ring is Cohen-Macaulay does not extend to non-Noetherian rings. 
We also show that for any valuation domain $V$ of finite Krull dimension, $V[x]$ is Cohen-Macaulay in the sense of Hamilton-Marley. 
All rings are commutative with unity, and all $R$-modules are unitary. \\

In commutative algebra and algebraic geometry, Cohen-Macaulayness of a ring or module is a desirable property.
If $R$ is a Noetherian local ring, we say that $R$ is {\it Cohen-Macaulay} if $\operatorname{depth} R = \dim R$, 
where $\operatorname{depth} R$ denotes the maximal length of a regular sequence contained in the maximal ideal of $R$.
In the Noetherian case, there are several characterizations of Cohen-Macaulay rings, and the notion of Cohen-Macaulayness is rather well understood.
In the general case, it is not clear what the ``right'' definition of a Cohen-Macaulay ring in the non-Noetherian case is. 
A direct extension of the definition to the non-Noetherian case seems to be a bit unnatural: 
For example, every valuation domain (which is not a field) has depth one, whereas the dimension can be arbitrarily large.
Valuation domains belong to the class of coherent regular rings, which when Noetherian, are well-known to be Cohen-Macaulay.
It is natural to search for a definition which generalizes the Noetherian case and includes meaningful non-Noetherian rings in the class of ``Cohen-Macaulay'' rings.
\\

S.\ Glaz initiated the study of this extension question. 
In \cite{Gla94}, she surveyed the ascent and descent properties of the extension $R^G \subset R$, where $G$ is a group acting on a commutative ring $R$, and $R^G$ is the ring of invariants. In section 4 of \cite{Gla94}, she proposed a conjecture:
\begin{quote} Let $R$ be a coherent regular ring, and let $G$ be a group of automorphisms of $R$. Assume that $R^G$ is a module retract of $R$ and that $R$ is a finitely generated $R^G$-module, then $R^G$ is a {\it Cohen-Macaulay} ring.
\end{quote}
As $R^G$ (and $R$) are not necessarily Noetherian (hence often not Cohen-Macaulay in the sense of $\depth R^G = \dim R^G$), she raised the question of finding a suitable extension of the definition of a Cohen-Macaulay ring which is not necessarily Noetherian.\\

Before explaining her definition and other generalizations, first we list characterizations of the Cohen-Macaulayness of a Noetherian local ring $R$. The following conditions are equivalent.
\begin{itemize}
\item [{[CM1]}] $R$ is Cohen-Macaulay.
\item [{[CM2]}] For every proper ideal $I$ of $R$, $\height I = \grade I$.
\item [{[CM3]}] For every proper ideal $I$ of $R$, if $\height I = \mu (I)$, then $I$ is unmixed.
\item [{[CM4]}]Every system of parameters of $R$ is a regular sequence.
\end{itemize}
Here, $\mu(I)$ denotes the minimal number of generators of $I$, and the {\it grade} of an ideal $I$ of $R$, denoted $\grade I$, is the length of a maximal regular sequence in $I$. 
These conditions will no longer be equivalent in the non-Noetherian case. Note that the notions of grade, unmixedness, and a system of parameters work well in the Noetherian case, but their useful properties do not hold, in general, in non-Noetherian rings.
For instance, in the Noetherian case, $\grade I > 0$ iff $0 :_R I = 0$, 
but this equivalence can fail to hold in the non-Noetherian case, cf. \cite[Example on p.\ 145]{Nor76}.
This issue can be fixed by considering the grade of $I$ in the polynomial extension $R[x]$. 

\begin{definition}[{\cite{Hoc74}}]
The {\it polynomial grade} of $I$, denoted $\pgrade I$, is 
\begin{equation}
\pgrade I = \sup \{ \grade IR[x_1,\dots,x_i] \mid i \in \mathbb{N} \}.
\end{equation}
\end{definition}
Note that $\grade I \le \pgrade I$, and equality holds if $R$ is Noetherian or $I$ is a complete intersection ideal. \\

Later, in \cite[Sec. 6]{Gla95}, Glaz defined an arbitrary ring $R$ to be {\it Cohen-Macaulay} if each  prime ideal $\mathfrak p$ of $R$, 
$\height \p = \pgrade \p$.
Her goal was to have a definition of a Cohen-Macaulay ring such that 
\begin{itemize}
\item [{[G1]}] the definition agrees in the Noetherian case, 
\item [{[G2]}] regular coherent rings are Cohen-Macaulay, and
\item [{[G3]}] certain invariant rings of regular coherent rings are Cohen-Macaulay. 
\end{itemize}


In \cite{Ham}, T.\ Hamilton continued the study of non-Noetherian Cohen-Macaulayness by investigating condition [CM3].
We say $I$ is {\it unmixed} if $\height I = \height \mathfrak \p$ for all associated primes $\p$ of $I$. 
A prime $\p$ is {\it associated} to $I$ if there exists $x \in R$ such that $\p = I :_R x$. 
In the Noetherian case, every proper ideal has a non-empty set of associated primes. 
This does not hold for non-Noetherian rings. 
A natural generalization of the notion of associated primes are the {\it weakly associated primes}, which are the primes $\p$ that are minimal over $I :_R x$ for some $x \in R$.
For any proper ideal $I$ in an arbitrary ring $R$, the set of weakly associated primes of $I$ is non-empty, 
and 
when the ring is Noetherian,
the weakly associated primes of $I$ and the associated primes of $I$ coincide. 
She defined $R$ to be {\it weakly-Bourbaki unmixed} if for each finitely generated ideal $I$ of $R$ with $\mu(I) \le \height I$, the weakly associated primes of $I$ are precisely the primes minimal over $I$.
In the same paper, she added two additional desirable conditions that should hold for $R$ to be Cohen-Macaulay.
\begin{itemize}
\item [{[H1]}] $R$ is Cohen-Macaulay iff $R[x]$ is Cohen-Macaulay;
\item [{[H2]}] $R$ is Cohen-Macaulay iff $R_\p$ is Cohen-Macaulay for all prime ideals $\p$.
\end{itemize}
With this notion of Cohen-Macaulayness, she was able to show that weak Bourbaki rings satisfy [G1] and that the if parts of both [H1],[H2] hold, see \cite[Theorems 1-3]{Ham}.\\

Most recently, Hamilton and Marley in \cite{HM} defined the notion of a strong parameter sequence which generalizes the notion of a system of parameters (see section 2 for the definition). 
They define an arbitrary ring $R$ to be {\it Cohen-Macaulay} if every strong parameter sequence of $R$ is a regular sequence of $R$.
With this definition, the authors were able to show the following:
\begin{itemize}
    \item Condition [G1] holds (\cite[Proposition 3.6, Proposition 4.2]{HM}).
    \item Condition [G2] holds locally (\cite[Theorem 4.8]{HM}). That is a local coherent regular ring is Cohen-Macaulay. (Thus, any valuation domain of finite Krull dimension is Cohen-Macaulay.)
    \item Condition [G3] holds if $\dim R \le 2$ (see \cite[Corollary 4.15]{HM} for the precise statement).
    \item The if parts of Hamilton's conditions [H1],[H2] hold (\cite[Corollary 4.6, Proposition 4.7]{HM}).
\end{itemize}
Furthermore, the two notions of Cohen-Macaulayness introduced by Glaz and Hamilton are both stronger condition than that of Hamilton-Marley cf.\ \cite[Proposition 2.6, Proposition 4.10]{HM}, 
and there are in fact examples showing that these notions of Cohen-Macaulayness need not be equivalent cf. \cite[Section 4]{AT}.
The primary goal of this paper is to study this definition of Cohen-Macaulay introduced by Hamilton and Marley.
In the sequel, unless otherwise stated, when we say a ring is Cohen-Macaulay, we always mean that it is Cohen-Macaulay in the sense of Hamilton and Marley.\\

Our main result is on the Cohen-Macaulayness of semigroup rings. 
A semigroup is called {\it affine} if it is finitely generated. 
Hence an affine semigroup ring is Noetherian.
In \cite[Example 4.9]{HM}, the authors show that the ring $\K + x \K [[x,y]]$ is non-Noetherian and Cohen-Macaulay, where $K$ is a field. 
This was extended to the semigroup rings $\K + x\K[x,y]$ in \cite[Theorem 4.10]{ADT}. 
Note that this semigroup ring is normal, and 
a normal affine semigroup ring is Cohen-Macaulay \cite{Hoc72}.
To this end, in \cite{ADT}, the authors proposed the following question.

\begin{question}[{\cite[Question 5.1]{ADT}}]\label{quesADTInt}
Let $R$ be a semigroup ring of finite Krull dimension which is not necessarily Noetherian. 
Then is $R$ Cohen-Macaulay if it is normal?
\end{question}
 
To answer this question, we study a family of rings $R = k + Q$, where $Q$ is an ideal of a finitely generated $k$-algebra $S$, where $k$ is a field. 
We note that this is a special case of amalgamated algebras in \cite{dAFF}. 
One of the challenging parts in studying the Cohen-Macaulayness in the sense of Hamilton and Marley is the verification of the condition of being a parameter sequence. 
In this setup, we give a method that can quickly verify this condition (\Cref{thmH2} and \Cref{lemChar}).
With this construction, we are able to answer \Cref{quesADTInt}, negatively.

\begin{theorem}[{\Cref{thmMain1}}]
The semigroup ring $K + xK[x,y,z]$ is normal. 
The sequence $xy,xz$ is a strong parameter sequence, but it is not a regular sequence. 
In particular, this ring is non-Noetherian and normal but not Cohen-Macaulay. 
\end{theorem}

Note that our theorem also addresses \cite[Remark 5.3]{ADT} negatively. 
Our result shows that Hochster's theorem fails to hold in the non-Noetherian case 
if one uses any of the notions of Cohen-Macaulayness mentioned thus far.
Considering this, it is natural to ask the following question. 
\begin{question}
Is there a notion of Cohen-Macaulayness satisfying [G1]-[G3] such that Hochster's theorem holds?
\end{question}

Our second result concerns the only if part of condition [H1]; that is, the question of $R[x]$ is Cohen-Macaulay if $R$ is Cohen-Macaulay. 
We have already mentioned that Hamilton and Marley \cite{HM} showed the converse. 
We answer the only if part of [H1] for the class of valuation domains (or Pr{\" u}fer domains) of finite Krull dimension.
Note that even though a valuation domain is always quasi-local, its Krull dimension can be infinite (in the Noetherian case, the Krull dimension of a local ring is finite).
We believe that this might be known to the experts, but to the best of our knowledge, 
 the statement and proof is not available. 

\begin{theorem}[{\Cref{thmV[x]}}]
Let $R$ be a Pr{\" u}fer domain of finite Krull dimension, then $R[x_1,\dots, x_n]$ is (locally) Cohen-Macaulay. In particular, $V[x_1,\dots,x_n]$ is Cohen-Macaulay if $V$ is a valuation domain.
\end{theorem}




\section{Cohen-Macaulayness in the sense of Hamilton and Marley}
In this section, we review the definition of a Cohen-Macaulay ring as in \cite{HM} and some preliminary results. \\

Let $R$ be a ring and $\underline{x} = x_{1},\ldots, x_{\ell}$ a sequence of elements of $R$. 
For any $n \in \mathbb{N}$, let $\underline{x}^{n}$ denote the sequence $x^{n}_{1},\ldots, x_{\ell}^{n}$. By $\mathbb{K}_{\bullet}(\underline{x})$, we mean the Koszul complex with respect to $\underline{x}$. 
For any $n \geq m$, there is a map of chain complexes 
\[ \varphi_{m}^{n} \colon \mathbb{K}_{\bullet}(\underline{x}^{n}) \to \mathbb{K}_{\bullet}(\underline{x}^{m}) \] 
induced by multiplication of $(x_{1} x_{2} \cdots x_{\ell})^{n-m}$ on $R$.  
A sequence of elements $\underline{x} = x_{1},\ldots, x_{\ell}$ of a ring $R$ is said to be \emph{weakly proregular}, 
if for each $m > 0$, there is an $n \geq m$ so that the maps \[ H_{i}( \varphi_{m}^{n}) \colon H_{i}(\mathbb{K}_{\bullet}(\underline{x}^{n})) \to H_{i}(\mathbb{K}_{\bullet}(\underline{x}^{m})) \] are $0$ for all $i \geq 1$. 
For $M$ an $R$-module,  
define the \emph{\u Cech complex} to be the complex \[ {\C}^{\bullet}_{\underline{x}}(M): 0 \to M \to \bigoplus_{i} M_{x_{i}} \to \bigoplus_{i < j} M_{x_{i}x_{j}} \to \ldots \to M_{x_{1}\cdots x_{\ell}} \to 0, \]
where the maps are natural up to a sign convention.  For each $i \in \mathbb{Z}$, the $i^{th}$ \emph{\u Cech cohomology} module of $M$ with respect to the sequence $\underline{x}$ is  $H^{i}_{\underline{x}}(M) := H^{i}({\C}^{\bullet}_{\underline{x}}(M))$. 
Schenzel \cite[Theorem 1.1]{weaklyproregular} showed that weakly proregular sequences $\underline{x}$ in a ring $R$ are precisely the sequences, where for $I = (\underline{x})R$, \u Cech cohomology $H_{\underline{x}}^{\bullet}(-)$ and local cohomology $H_{I}^{\bullet}(-) := \varinjlim_{n} \text{Ext}_{R}^{
 \bullet}(R/I^{n},-)$ are canonically isomorphic functors. 
 These two functors are isomorphic in the Noetherian case, but they are not isomorphic in general, cf. \cite[paragraph after Theorem 2.3]{HM}.
For any sequence of elements $\underline{x} = x_1,\dots, x_s$, we define $\ell (\underline{x})$ to be the length of the sequence $\underline{x}$.


\begin{definition}[{\cite[Definition 3.1]{HM}}]\label{defPar}
Let $R$ be a ring. 
A sequence of elements $\underline{x} = x_{1},\ldots, x_{\ell}$ is said to a \emph{parameter sequence} on $R$ if the following conditions hold:
\begin{enumerate}
\item $(\underline{x})R \neq R.$

\item $\underline{x} \text{ is weakly proregular}$. 

\item $H_{\underline{x}}^{\ell}(R)_{\mathfrak{p}} \neq 0 \text{ for each }\mathfrak{p} \in V_R(\underline{x}R).$
\end{enumerate}
Here, $V_R(I)$ denotes the set of prime ideals of $R$ containing the ideal $I$, and for a sequence of elements $f_1,\dots, f_s \in R$,  we define $V_R(f_1,\dots, f_s) := V_R ( (f_1,\dots, f_x)R)$.
Moreover, we say $\underline{x}$ is a \emph{strong parameter sequence} on $R$ if $x_{1},\ldots, x_{i}$ is a parameter sequence on $R$ for all $1 \leq i \leq \ell$.  
\end{definition}

Note that there is an example \cite[Proposition 2.9]{HM} of a parameter sequence that is not a strong parameter sequence. The next proposition states that a strong parameter sequence is part of a system of parameters if the ring is local and Noetherian.

\begin{prop}[{\cite[Proposition 3.6]{HM}}]\label{propSopHt}
Let $R$ be a ring, and let $\underline{x}$ be a parameter sequence of $R$. Then $\height (\underline x)R \ge \ell(\underline{x})$. 
\end{prop}

\begin{definition}[\cite{HM}]
We say that a ring $R$ is \emph{Cohen-Macaulay} if every strong parameter sequence on $R$ is a regular sequence on $R$.
\end{definition}

Hamilton-Marley showed the following characterizations of their notion of Cohen-Macaulayness. 
\begin{theorem}[{\cite[Proposition 4.2]{HM}}]\label{charHM}
Let $R$ be a ring. The following conditions are equivalent:
\begin{enumerate}[\indent\;$(1)$]
\item $R$ is Cohen-Macaulay.
\item $\grade (\underline{x}) R = \ell(\underline{x})$ for every strong parameter sequence $\underline{x}$ of R.
\item $\pgrade (\underline{x})R = \ell(\underline{x})$ for every strong parameter sequence $\underline{x}$ of R.
\item $H_{\underline{x}}^i (R) = 0$ for all $i < \ell(\underline{x})$ and for every strong parameter sequence $\underline{x}$ of $R$.
\item $H_i (\underline{x}; R) = 0$ for all $i \ge 1$ and for every strong parameter sequence $\underline{x}$ of $R$. Here, $H_i (\underline{x}; R)$ denotes the Koszul homology module with respect to the sequence $\underline{x}$. 
\end{enumerate}
\end{theorem}

\begin{remark}[{\cite[Proposition 3.3(f)]{HM}}] \label{regimpliesparameter} In any ring, every regular sequence on $R$ is a parameter sequence on $R$.
\end{remark}


Lastly, we record some results on local cohomology modules which we will use in the following sections. For the proofs of the following results and basic facts on local cohomology, we refer the read to \cite{BS,24HR}.

\begin{prop}\label{propPre} Let $R$ be a Noetherian ring, and let $I,J$ be $R$-ideals.
\begin{enumerate}[$(a)$]
\item {\textup{[Mayer-Viertories sequence]}}  One has the following long exact sequence of local cohomology modules
\begin{align*}
0 &\to H^0_{I+J} (R) \to H^0_I (R) \oplus H^0_J (R) \to H^0_{I \cap J} (R) \\
&\to H^1_{I+J} (R) \to H^1_I (R) \oplus H^1_J (R) \to H^1_{I \cap J} (R) \\
&\to H^2_{I+J} (R) \to \cdots.
\end{align*}

\item $\grade (I,R) = \inf \{ i \mid H^i_I(R) \neq 0 \}$. In particular, $H^i_I(R) = 0$ for all $i < \grade (I,R)$. 

\item If $R$ is local and $I$ is an ideal generated by a system of parameters of $R$, then $H^{\dim R}_I (R) \neq 0$.

\end{enumerate}
\end{prop}

\section{Cohen-Macaulayness of $k + Q$ rings}
In this section, we will study the Cohen-Macaulayness of rings of type $k + Q$, where $k$ is a field and $Q$ is an ideal in a Noetherian ring. 
We first state a lemma which will be useful in our setup. 

\begin{lem}\label{lem5terms}
Let $R \subset S$ be a ring extension, and suppose $\underline{f} = f_{1}, \ldots, f_{d} \in (R :_{S} S)$. Then the following hold: 
\begin{enumerate}[$(a)$]
\item We have an exact sequence of {\u C}ech cohomology modules
\begin{equation}\label{eq5terms}
0 \to H^0_{\underline{f}} (R) \to H^0_{\underline{f}} (S) \to H^0_{\underline{f}} (S/R) \to H^1_{\underline{f}} (R) \to H^1_{\underline{f}} (S) \to 0, 
\end{equation}
and $H^{i}_{\underline{f}}(R) = H^{i}_{\underline{f}}(S)$ for each $i \geq 2$. 
\item $\underline{f}$ is a weakly proregular sequence in $R$ if and only if $\underline{f}$ is a weakly proregular sequence in $S$. 
\end{enumerate}
\end{lem}
\begin{proof}
The statement $(a)$ follows from the long exact exact sequence of {\u C}ech cohomology modules induced from the short exact sequence of $R$-modules $0 \to R \to S \to S/R \to 0$ 
and from the fact that $(\underline{f})(S/R) = 0$ implies $H^{i}_{\underline{f}}(S/R) = 0$ for $i \geq 1$.\\ 

For $(b)$, the short exact sequence $0 \to R \to S \to S/R \to 0$ yields a short exact sequence of inverse systems of Koszul homology modules for any $i \geq 0$: \[ 0 \to \{  H_{i}(\underline{f^{n}} ; R)\} \to \{ H_{i}(\underline{f^{n}} ; S)\}  \to \{ H_{i}(\underline{f^{n}} ; S/R)\}  \to 0. \]  

For $n \geq m$, the multiplication map $K_{\bullet}( \underline{f^{n}} ; S/R) \overset{ (f_{1} \cdots f_{d})^{n-m}}\to  K_{\bullet}( \underline{f^{n}} ; S/R)$ is identically zero, so that $\{ H_{i}(\underline{f^{n}} ; S/R)\}$ is pro-zero, i.e., $H_{i}(\underline{f^{n}} ; S/R) = 0$ for $i > 0$. 
Thus, by \cite[Remark 2, p. 24]{Gro}, 
for each $i \geq 0$, the inverse system  $\{  H_{i}(\underline{f^{n}} ; R)\}$ is pro-zero if and only if the inverse system $\{  H_{i}(\underline{f^{n}} ; S)\}$ is pro-zero.
\end{proof}

\noindent{\bf Setup:}
We will use the following setup for the rest of the section.  
Let $S$ be a domain which is essentially of finite type over a field $k$. 
That is, $S$ is a localization of a finitely generated $k$-algebra. 
In particular, $S$ is a Noetherian ring.
An $S$-ideal $Q$ gives rise to a sub $k$-algebra $R = k +Q \subset S$. 
We call $R$ the {\it amalgamated $k$-algebra} of the ring $S$ along the ideal $Q$\footnote{The names comes from \cite{dAFF}. According to their terminologies, $R$ is the amalgamated $k$-algebra of $S$ along $Q$ under the embedding $k \to S$.}.
We will consider the case where $Q$ is at most two-generated.

\begin{remark}
We present examples of Noetherian and non-Noetherian amalgamated $k$-algebras. 
Let $S = k[x,y]$. 
 For $Q = (x,y^2)S$, $R = k + Q = k[x,xy,y^2,y^3]$ is Noetherian.
 For $Q = (x)S$, $R = k + Q = k[x, xy, xy^2, xy^3, \dots]$ is not Noetherian.
\end{remark}

\begin{remark}
For any ring extension $A \subset B$ of domains having the same field of fractions $K$, we call $\mathfrak{c}_{B/A} := A:_K B$ the {\it conductor} of the extension. 
It it well-known that $\mathfrak{c}_{B/A} = A :_B B = A :_A B$, and $\mathfrak{c}$ is the largest $B$-ideal contained in $A$. 
Furthermore, any prime ideal $\p$ of $A$, $(\mathfrak{c}_{B/A})_\p \subset \mathfrak{c}_{B_\p/ A_\p}$.
With the setup above, if $R \neq S$, then $Q$ is the conductor of the extension $R \subset S$, i.e., $R :_R S = Q$. 
As a consequence, for any prime ideal $\p \neq Q$, $Q_\p = R_{\p} \subset \mathfrak{c}_{S_\p/ R_\p}$, so $R_\p = \mathfrak{c}_{S_\p/ R_\p}$, i.e., $R_\p = S_\p$. 
\end{remark}

\begin{prop}\label{propConstruction}
Let $R, S$, and $Q$ be as above. If $Q \neq S$, then we have the following. 
\begin{enumerate}[$(a)$]
\item $Q$ is a maximal ideal of $R$.

\item For any prime ideal $\p \in \Spec R \setminus \{ Q \}$, $R_\p$ is Noetherian. 

\item Assume that $S$ is Cohen-Macaulay. Then $R$ is Cohen-Macaulay in the sense of Hamilton-Marley if $R_Q$ is Cohen-Macaulay in the sense of Hamilton-Marley.
\end{enumerate}
\end{prop}

\begin{proof}
(a): Since $R/Q  \cong k$ is a field, $Q$ is a maximal ideal.\\ 

(b): 
If $\p \ne Q$, then $R_\p :_{R_\p} S_\p = (R:_R S)_\p = Q_\p = R_\p$. So, $R_\p = S_\p$ is Noetherian since $S$ is Noetherian.\\

(c): This follows from \cite[Proposition 4.7]{HM}.
\end{proof}

We note that we do not know if $R_Q$ is Cohen-Macaulay under the condition of $R$ and $S$ being Cohen-Macaulay. This is a special case of an implication of condition [H2]. 

\begin{question} Is $R_Q$ Cohen-Macaulay if $R$ and $S$ are Cohen-Macaulay?
\end{question}

We first treat the case where $Q$ is generated by two elements in $S$. 

\begin{theorem}\label{thmH2}
Let $S$ be a domain which is essentially of finite type over a field $k$, $Q$ an $S$-ideal, and $R = k + Q$. 
If $Q = (f,g) \subsetneq S$ for some nonzero elements $f,g$ of $S$, then 
$f,g$ is a strong parameter sequence in $R$ if and only if 
$H^{2}_{f,g}(S) \neq 0$. 
Furthermore, if $R$ is Cohen-Macaulay, then $H^{2}_{f,g}(S) = 0$.
\end{theorem}

\begin{proof}
As $f$ is a regular element of $R$, it is a parameter sequence. 
Observe that since $Q = (R :_{S} S)$ is the conductor ideal and $S$ is Noetherian, by \Cref{lem5terms}(b), $f,g$ is a weakly proregular sequence in $R$. 
Thus, to prove the first statement, we show that for $f,g$ to be a parameter sequence, it is equivalent to say $H^{2}_{ f,g}(R)_{\mathfrak{p}} \neq 0$ for each $\mathfrak{p} \in V_R( f,g)$.  
Note that $V_R(f,g) = \{ Q \}$ by \cite[Prop. 2.6(2)]{dAFF}.  
Thus, $H^{2}_{f,g}(R) \neq 0$ if and only if $(H^{2}_{f,g}(R))_Q \neq 0$. 
Then the equivalence is clear since $H^2_{f,g}(R) = H^2_{f,g}(S)$ by \Cref{lem5terms}(a).\\

Now, we show the second statement. 
Suppose by way of contradiction that $f,g$ is a strong parameter sequence on $R$. Since $R$ is Cohen-Macaulay, $f,g$ is a regular sequence on $R$.
Thus, $H^{1}_{f,g}(R) = 0$ by \Cref{propPre}(b).
By \Cref{lem5terms}(a), we have the exact sequence  
\begin{equation*} 
H^{0}_{f,g}(S) \to H^{0}_{f,g}(S/R) \to H^{1}_{f,g}(R). 
\end{equation*}
Since $\grade (f,g)S > 0$, $H^{0}_{f,g}(S) = 0$ by \Cref{propPre}(b), 
and since $R \neq S$ and $Q = R:_S S$, one has $H^{0}_{f,g}(S/R) \neq 0$. 
Therefore, $H^{1}_{f,g}(R) \neq 0$, but this is a contradiction. 
\end{proof}

We do not know if the converse of the last statement is true. It is natural to ask if it is true when $S$ is a polynomial ring. 

\begin{question}
Let $S = k[x_{1}, \ldots, x_{n}]$ be a polynomial ring over a field $k$, $Q$ an $S$-ideal, and $R = k + Q$. 
Let $Q = (f,g) \subsetneq S$ for some nonzero elements $f,g$ of $S$.
Is the ring $R$ Cohen-Macaulay if $H^2_{f,g} (S) = 0$? 
\end{question}

\begin{remark}
It is worth noting the contrapositive of \Cref{thmH2}. 
That is, if $H^2_{f,g} (S) \neq 0$, then $R$ is not Cohen-Macaulay. 
We will use this fact to give examples of non Cohen-Macaulay rings.

\end{remark}

In \Cref{lemChar} below, we give an ideal-theoretic characterization of the condition $H^{2}_{f,g}(S) \neq 0$, 
and this provides a quick way to verify the condition of $H^{2}_{f,g} (S) \neq 0$. 
The essential idea of this criterion appeared in \cite{HKM}, and the equivalence of $(a)$ and $(c)$ in \Cref{lemChar} is due to them, \cite[Prop. 2.6]{HKM}. 
For an $S$-ideal $I$, let $\ara I := \min \{ \ell \in \mathbb{Z}_{\ge 0} \mid \text{there exist } z_1,\dots,z_\ell \in S, \sqrt{z_1,\dots, z_\ell} = \sqrt{I} \}$ denote the \textit{arithematic rank} of the ideal $I$.

\begin{lemma}\label{lemChar}
Let $S$ be a Noetherian UFD and $f,g$ non-zero elements in S. Let $d = \gcd (f,g)$. 
Then we have 
\begin{equation*}
H^2_{f,g}(S) = H^2_{f,g}(S_d) = H^2_{\frac fd,\frac gd}(S_d).
\end{equation*} 
Furthermore, the following statements are equivalent:
\begin{enumerate}[\indent$(a)$]
\item $H^2_{f,g}(S) = 0$.
\item $\ara (f/d, g/d)S = 1$ or $d$ is in the radical of $(f/d,g/d)S$.
\item $(fg)^n \in (f^{n+1}, g^{n+1})$ for some positive integer $n$.
\end{enumerate}
In particular, $H^2_{f,g}(S) \neq 0$ iff $\ara (f/d, g/d)S = 2$ and  $(f/d,g/d)S_d \neq S_d$. 
\end{lemma}
\begin{proof}
Note that $H^2_{f,g}(S)$ is the cokernel of the map
\begin{equation}
\pi: S_f \oplus S_g \to S_{fg}.
\end{equation}
Since $d \mid f$, we have a natural map $S_d \to S_f = (S_d)_f$, and similarly for $S_d \to S_g$. 
Hence $\operatorname{coker} \pi = \operatorname{coker} \pi_d$, where $\pi_d: (S_d)_f \oplus (S_d)_g \to (S_d)_{fg}$.
This shows the first equality.
The second equality follows from \cite[Porp 2.1(e)]{HM} since $\sqrt {(f,g)S_d} = \sqrt{(f/d,g/d) S_d}$.\\

Next, we show the equivalence. 
The equivalence of $(a)$ and $(c)$ is \cite[Prop 2.6]{HKM}.
We will show that $H^2_{\frac fd,\frac gd}(S_d) = 0$ iff $\ara (f,g)S =1$ or $d$ is in the radical of $(f/d,g/d)$ in $S$. 
If $\ara (f,g)S = 1$, then $H^2_{f, g}(S) = 0$ by definition. 
If $d \in \sqrt{(f/d,g/d)}$, then $(f/d,g/d)S_d = S_d$.
Thus, $H^2_{\frac fd,\frac gd} (S_d) = H^2_{1} (S_d ) = 0$. 
Conversely, suppose $H^2_{\frac fd,\frac gd}(S_d) = 0$. 
It suffices to show that if $H^2_{\frac fd,\frac gd}(S_d) = 0$ and $\ara (f,g) \neq 1$, then $\ara (f,g)S = 2$ and $d$ is in the radical of $(f/d,g/d)$ in $S$. 
Since $(f,g)S$ is two generated, $\ara (f,g)S \le 2$, and since $f,g$ are not zero in an integral domain, $\ara (f,g)S \ge 1$. Thus, we have $\ara (f,g)S = 2$.
Let 
$J' = (f/d,g/d)S$. 
Assume to the contrary that $d$ is not in the radical of $J'S$. 
Then $J'S_d$ is a proper ideal of $S_d$. 
We claim that $\height J'S_d = 2$. 
Since $S_d$ is Noetherian, $\height J'S_d \le 2$ and since $J'S_d \neq 0$, $\height J'S_d \ge 1$. 
Hence, the claim follows immediately once we have shown that $J'S_d$ is not contained in any prime ideal of height $1$.
Since $S$ is a UFD, $S_d$ is a UFD, and under $S \to S_d$ $f/d,g/d$ remain as relatively prime elements. 
Hence $\gcd(f/d,g/d) = 1$ in $S_d$, and there is no principal prime ideal containing the ideal $J'S_d$. 
Therefore, $\height J'S_d = 2$. \\
 
Let $\p$ be a minimal prime of $\height J'S_d$ of height $2$.
Hence, $(H^2_{\frac fd,\frac gd} (S_d) )_\p = H^2_{\frac fd, \frac gd} (S_d)_\p \neq 0$ by \Cref{propPre}(c), and this is a contradiction. 
This completes the proof.
\end{proof}

\begin{ex} Let $S = k[x,y]$, where $k$ is a field, and let $Q = (x,y^2)S$. Then $\ara (x,y^2) = 2$ and $\gcd (x,y^2) = 1$ is not in $(x,y) = \sqrt{(x,y^2)}$. 
By \Cref{lemChar}, $H^2_Q (S) \neq 0$. 
Therefore, by \Cref{thmH2}, the ring $k + Q = k[x,xy,y^2,y^3]$ is not Cohen-Macaulay. 
\end{ex}

\begin{ex} Let $S = k[x,y,z]$. 
Consider ideals $Q_1 = (x,y)S, Q_2 = (x,yz)S, Q_3 = (xy,xz)S$.
Define $R_i = k + Q_i$ for $i = 1,2,3$. 
Then $R_i$ are not Cohen-Macaulay. \\

By \Cref{thmH2}, it suffices to show that $H^2_{Q_i} (S) \neq  0$.
We apply \Cref{lemChar} to $Q_1,Q_2,Q_3$. 
That is $Q_i S_{\gcd} \neq S_{\gcd}$. 
For $Q_1$ and $Q_2$, notice that $\gcd(x,y) = \gcd(x,yz) = 1$, and 
for $Q_3$, $\gcd (xy,xz) = x$ and $(xy,xz)S_x = (y,z)S_x \neq S_x$. 
Therefore, the rings $R_1, R_2, R_3$ are not Cohen-Macaulay. 
\end{ex}

Next, we treat the case where $Q$ is principal. 

\begin{lem}\label{lemNotCM}
Let $S$ be a domain which is essentially of finite type over a field $k$ and $Q = fS$ for some nonzero element $f$ in $S$. 
If $f,g,h \in S$ is a regular sequence on $S$, then $fg, fh$ is a strong parameter sequence of $R$ that is not a regular sequence of $R$.
In particular, the ring $R := k + QS$ is not Cohen-Macaulay. 
\end{lem}

\begin{proof}
First, we claim that $fg, fh$ is a strong parameter sequence on $S$ that is not regular. 
By \Cref{lem5terms}(b), it is clear that $fg, fh$ is a weakly proregular sequence since $S$ is Noetherian.  
We show that for each $\mathfrak{p} \in V_R (fg, fh)$, we have that $H^{2}_{fg,fh}(R)_{\mathfrak{p}} \neq 0$. 
First, suppose that $\mathfrak{p} \neq Q$. 
Then $f \notin \mathfrak{p}$ ($\sqrt{fR} = Q$), so that 
\[ H^{2}_{fg,fh}(R)_{\mathfrak{p}} = H^{2}_{g,h}(R_{\mathfrak{p}}) = H^{2}_{g,h}(S_{\mathfrak{p}}) \neq 0,            \] 
since the regular sequence $f,g$ of $S$ extends to a regular sequence in $S_{\mathfrak{p}} = R_{\mathfrak{p}}$. 
It remains to show $H^{2}_{fg,fh}(R)_{Q} \neq 0$. 
Choose a prime $\mathfrak{q} \in V_S( f,g,h)$. 
Thus $\mathfrak{q} \cap R = Q$ by \cite[Prop 2.6(2)]{dAFF}. 
Since $H^{2}_{fg, fh}(R)_{Q} = H^{2}_{fg, fh}(S)_{Q}$ by \Cref{lem5terms}(a), 
it is enough to show that $H^{2}_{fg,fh}(S)_{\mathfrak{q}} \neq 0$, since $R - Q \subseteq S - \mathfrak{q}$. 
We note that since $(f,g,h)$ is a regular sequence of length $3$, one has $(f) \cap (g,h) = (fg,fh)$.
From the Mayer-Vietoris sequence, we have the following exact sequence   
\begin{equation*}
        H^{2}_{fg, fh}(S_\mathfrak{q}) \to  H^{3}_{f,g,h}(S_\mathfrak{q})  \to H^{3}_{g,h}(S_\mathfrak{q}) \oplus H^{3}_{f}(S_\mathfrak{q}) = 0.   
\end{equation*}
Here, the right most cohomology modules are zero because $3 > \ara {(g,h)}, \ara{(f)}$.
Since $S_{\mathfrak{q}}$ is a Noetherian ring of dimension $3$, by \Cref{propPre}(c), $H^{3}_{f,g,h}(S_{\mathfrak{q}}) \neq 0$. Thus $H^{2}_{fg,fh}(S_{\mathfrak{q}}) \neq 0$, and so $fg, fh$ is a strong parameter sequence in $R$. \\

Now, we show that $fg, fh$ is not a regular sequence. 
From the Mayer-Vietoris sequence, we have the following exact sequence
\[  H^{1}_{f,g,h}(S)  \to H^{1}_{g,h}(S) \oplus H^{1}_{f}(S) \to H^{1}_{fg, fh}(S).               
\] 
Since $\text{grade}( (f,g,h)S) \geq 2$, $H^{1}_{f,g,h}(S) = 0$ by \Cref{propPre}(c), then $H^{1}_{f}(S) \neq 0 \Rightarrow H^{1}_{fg,fh}(S) \neq 0$. 
By \Cref{lem5terms}(a), we have the surjective map $H^{1}_{fg,fh}(R) \twoheadrightarrow H^{1}_{fg,fh}(S)$.  
Thus, $H^{1}_{fg, fh}(R) \neq 0$ and $\pgrade( (fg,fh)R, R) =1$ by \cite[Prop. 2.7]{HM}. 
By \Cref{charHM}, $R$ is not Cohen-Macaulay.
\end{proof}

\begin{ex}\label{examBad}
Let $S = k[x,y,z]$, and consider the regular sequence $x,y,z$ on $S$. 
Then by \Cref{lemNotCM}, the ring $R = k + xk[x,y,z]$ is not Cohen-Macaulay, as $xy,xz$ is a strong parameter sequence that is not a regular sequence. 
It is worth mentioning that the ring $T = k + xk[x,y]$ \emph{is} Cohen-Macaulay. 
\end{ex}

\begin{theorem}
Let $S$ be a domain which is essentially of finite type over a field $k$, and let $f$ be a nonzero element in $S$. 
If $S$ is Cohen-Macaulay and $\dim S \ge 3$, then the ring $k + fS$ is not Cohen-Macaulay. 
\end{theorem}

\begin{proof}
By \Cref{lemNotCM}, it suffices to show there exists $g,h \in S$ such that $f,g,h$ is a regular sequence of length $3$.
Notice that $S/fS$ is Cohen-Macaulay. 
Hence $\Min (S/fS) = \Ass (S/fS)$. 
Since these sets are finite, by the prime avoidance lemma, one can find an element $g$ which is not in the union of the associated primes of $S/fS$. 
Then $S/(f,g)S$ is Cohen-Macaulay. 
Thus one can find $h$, similarly. 
\end{proof}

\section{an Example of a non Cohen-Macaulay normal semigroup ring}
Let $H$ be a semigroup in $\mathbb{Z}^{n}$. We say that $H$ is \emph{normal} if for any $m \in \mathbb{N}$ and $s \in \mathbb{Z}H$, we have $ms \in H \Rightarrow s \in H$. A celebrated theorem of Hochster \cite{Hoc72} states that if $H$ is a finitely generated normal semigroup in $\mathbb{Z}^{n}$, then $k[H]$ is Cohen-Macaulay, where $k$ is any field. 
In this section, we will present an example of a semigroup ring $k[H]$, where the ring $K[H]$ is normal (so, the semigroup $H$ is normal\footnote{We do not know if $H$ is normal is sufficient to conclude that $K[H]$ is normal.}), yet $k[H]$ is not Cohen-Macaulay. It is necessary that $H$ is not finitely generated by the result of Hochster.
For more details on (affine) semigroup rings, we refer the reader to \cite{Hoc72} or \cite[Section II.6]{BH}. \\

Recall the question we mentioned in the introduction.
\begin{question}[{\cite[Question 5.1]{ADT}}]
Let $R$ be a semigroup ring of finite Krull dimension which is not necessarily Noetherian. 
Then is $R$ Cohen-Macaulay if it is normal?
\end{question}

We answer this question negatively.

\begin{theorem}\label{thmMain1}
 Let $S := k[x,y,z] = k[\mathbb{N}^3_0]$ and $R := k + xS = k[H]$, where $k$ is a field and $H = \mathbb{N}_0 \times \mathbb{N}^{2} \cup \{(0,0,0)\}$. Then we have the following.
\begin{enumerate}[\indent$(a)$]
\item The semigroup ring $k[H]$ is normal.
\item $R$ is a non-Noetherian subring of $S$ of dimension $3$.
\item $xy,xz$ is a strong parameter sequence of $R$ that is not a regular sequence of $R$. 
\end{enumerate}
In particular, $k[H]$ is a non-Noetherian normal domain which is not Cohen-Macaulay in the sense of Hamilton-Marley.
\end{theorem}

\begin{proof}
(a): Notice that since $S$ is a UFD, $S$ is integrally closed.
Let $\overline{\phantom{A}}$ denote the integral closure of a domain in its field of fractions. 
Suppose that $f \in \overline{R}$. Then since $R \subseteq S$, we have that $\overline{R} \subseteq \overline{S} = S$, so that $f \in S$. Now, for some $v \in \mathbb{N}$ and $g_{1},\ldots, g_{v} \in R$, we have an equation of the form 
\begin{equation}\label{eqIntegral}
f^{v} + g_{1}f^{v-1} + \cdots + g_{v-1}f + g_{v} = 0.       
\end{equation}
For each $i \in \{1,\ldots, v\}$, write  $g_{i} = a_{i} + xh_{i}$, where $a_{i} \in k$, and $h_{i} \in S$. 
Going modulo $xS$, \Cref{eqIntegral} induces an equation of integrality of the image of $f$ in $S/xS$ over $R/(xS \cap R) = k$. 
Since $k$ is a field, it is integrally closed. 
Therefore $f \in k + xS = R$. Hence $R$ is integrally closed. \\

(b): The fact that $\dim R = 3$ follows from a result of Gilmer, \cite[Proposition 3.3]{And06}. 
Suppose $R$ is Noetherian. Then by Krull's principal ideal theorem, $\height \m = 1$. Let $\n = (x,y,z)S$. Since $\n \cap R = \m$, the dimension formula \cite[Theorem 15.5]{Mat} implies that 
$\height \n \le \height \m$. But this is a contradiction. \\ 

(c): Notice that $S$ is Cohen-Maculay and $x,y,z$ forms a regular sequence in $S$. Thus by \Cref{lemNotCM}, $xy,xz$ is a strong parameter sequence of $R$ which is not a regular sequence of $R$ (cf. \Cref{examBad}).
\end{proof}

Let $k$ be a field. We showed that $R = k+xk[x,y,z]$ is a non-Noetherian normal domain, but $R$ is not Cohen-Macaulay as $xy, xz$ is a strong parameter sequence which is not a regular sequence.
On the other hand, the ring $A = k + xk[x,y]$ is Cohen-Macaulay. 
One might be tempted to say that $A[z]$ is not Cohen-Macaulay since one has the sequence $xy,xz$ in $A[z]$.
However, $xy,xz$ is not a strong parameter sequence in $A[z]$, and we do not know whether or not $A[z]$ is Cohen-Macaulay. 
In the following section, we discuss the Cohen-Macaulayness of $R[x]$, where $x$ is an indeterminate over $R$.

\section{A remark on condition [H1]}
Recall Condition [H1]: $R$ is Cohen-Macaulay iff $R[x]$ is Cohen-Macaulay, where $x$ is an indeterminate over $R$. 
The if part was shown in \cite[Corollary 4.6]{HM}, 
but it is not known that $R[x]$ is Cohen-Macaulay if $R$ is Cohen-Macaulay. 
Unlike in the Noetherian case, several properties of the ring $R$ do not pass to its polynomial extension $R[x]$. 
In particular, $R[x]$ need not be coherent even if $R$ is coherent. 
In this section, we collect a few results to prove that $R[x]$ is Cohen-Macaulay if $R$ is a Pr{\" u}fer domain of finite Krull dimension. 
We believe that \Cref{thmV[x]} might be known to experts, but it was not written anywhere.
We note that there are examples of coherent regular quasi local rings such that $R[x]$ is not even coherent, see \Cref{examNotCoh}.  
Here, we list some definitions, theorems, and a remark we will use in the proof of \Cref{thmV[x]}.\\

\begin{remark}\label{remarkV[x]}
Let $R$ be a commutative ring.
\begin{enumerate}[\indent$(a)$]  
\item {\cite[Prop. 3]{Sab76}} If $R$ is a Pr{\" u}fer domain, then $R[x_1,\dots,x_n]$ is coherent for any $n$. 

\item {\cite[p.122]{CE}} The {\it weak dimension} of $R$, denoted $\operatorname{w.dim} R$, is 
\begin{equation*}
\sup \{ i \in \mathbb{Z}_{\ge 0} \mid \operatorname{Tor}_i^R (M,N) \neq 0, M,N~R\text{-modules} \}.
\end{equation*}
Furthermore, $\operatorname{w.dim} R \le \operatorname{gl.dim} R := \sup \{ i \in \mathbb{Z}_{\ge 0} \mid \operatorname{Ext}^i_R (M,N) \neq 0, M,N~R\text{-modules} \}.$

\item {\cite[Lemma 3]{Gla89}}
If $(R,\m)$ is a quasi-local coherent regular ring, then $\grade (\m,R) = \operatorname{w.dim} R$. In particular, $\operatorname{w.dim} R$ is finite.


\item {\cite[Theorem 0.14]{Vas76}} If $R[x_1,\dots, x_d]$ is coherent, then $\operatorname{w.dim} R[x_1,\dots, x_d] = \operatorname{w.dim} R + d$.

\item If $R$ is a Pr{\" u}fer domain of finite Krull dimension, then $\operatorname{w.dim} R < \infty$.

\item {\cite{Gla87}} A coherent ring of finite weak dimension is a regular ring, although not every coherent regular ring has finite weak dimension (e.g., $k[[x_1,x_2,\dots]]$, where $k$ is a field).
\end{enumerate}
\end{remark}

\begin{theorem}\label{thmV[x]}
Let $R$ be a Pr{\" u}fer domain of finite Krull dimension, then $R[x_1,\dots, x_n]$ is (locally) Cohen-Macaulay. In particular, $V[x_1,\dots,x_n]$ is Cohen-Macaulay if $V$ is a valuation domain.
\end{theorem}
\begin{proof}
Since $R$ is a Pr{\" u}fer domain of finite Krull dimension, $\operatorname{w.dim} R < \infty$  (\Cref{remarkV[x]}(e)). 
By \Cref{remarkV[x]}(a), $R[x_1,\dots,x_n]$ is coherent. 
Therefore, \Cref{remarkV[x]}{(d)} implies that $\operatorname{w.dim} R[x_1,\dots,x_n] = \operatorname{w.dim} R + n < \infty$. 
Thus, $R[x_1,\dots,x_n]$ is a coherent ring with finite weak dimension. 
So, it is coherent regular by \Cref{remarkV[x]}(f).
Then we are done by \cite[Theorems 4.7, 4.8]{HM}.
\end{proof}

We end the paper with an example of Soublin.  
He constructed an example of a $2$ dimensional coherent regular ring $R$ such that $R[x]$ is not coherent. Alfonsi used this example to construct a $2$ dimensional 
coherent quasi-local regular domain such that $R[x]$ is not coherent. 
It is an interesting question whether $R[x]$ is Cohen-Macaulay.
We do not have an answer to this question.

\begin{example}[{\cite[\S 7.1.13]{Gla71},\cite{Sou70},\cite{Alf81}}]\label{examNotCoh}
Let $\mathbb{N}$ and $\mathbb{Q}$ denote the set of natural numbers and the set of rational numbers, 
$S = \mathbb{Q}[[ t, u]]$ a power series ring in two variables over $\mathbb{Q}$, and $R = \prod_{\alpha \in \mathbb{N}} S_\alpha$, where $S_\alpha \cong S$. 
Then $R$ is a coherent ring of weak dimension $2$, but $R[x]$ is not coherent. 
\end{example}

\begin{question}
Is $R[x]$ Cohen-Macaulay if $R$ is the ring in the example of Soublin and Alfonsi?
\end{question}
\bigskip
\newcommand{\etalchar}[1]{$^{#1}$}

\end{document}